\documentclass[a4paper]{article}

\usepackage[T1]{fontenc}
\usepackage[utf8]{inputenc}
\usepackage[pass]{geometry}
\usepackage{tikz}
	\usetikzlibrary{positioning}
\usepackage{ragged2e}
\usepackage{amsmath}
\usepackage{amsfonts}
\usepackage{mathtools}
\usepackage{amsthm}
\usepackage{noindentafter}
\usepackage{caption}
\RequirePackage{hyperref}
	\hypersetup{linktoc=all,hidelinks}
\usepackage[capitalize,nameinlink]{cleveref}

\newtheorem{theorem}{Theorem}
\newtheorem{oldtheorem}{Theorem}
\newtheorem{lemma}[theorem]{Lemma}
\newtheorem{corollary}[theorem]{Corollary}
\newtheorem{observation}[theorem]{Observation}
\newtheorem{conjecture}[theorem]{Conjecture}
\newtheorem{claim}{Claim}[subsection]
\theoremstyle{remark}
\newtheorem{remark}{Remark}

\crefname{claim}{Claim}{Claims}

\crefname{case}{Case}{Cases}

\let\eqref\labelcref

\crefname{oldtheorem}{Theorem}{Theorems}

\crefname{observation}{Observation}{Observations}

\newcommand{\join}{\vee}
\newcommand{\seqjoin}{\join}
\newcommand{\pathto}[2]{\ensuremath{#1\mathord{-}#2}}
\newcommand{\complement}{\overline}
\newcommand{\dir}{\vec}
\newlength{\revdirraise}
\newcommand{\revdir}[1]{
	\settoheight{\revdirraise}{#1}
	\ooalign{$#1$\cr{\raisebox{\revdirraise+0.85pt}
		{\rlap{$\mkern4.1mu$\rotatebox[origin=c]{180}
			{$\vec{\phantom #1}$}}}}}}

\DeclarePairedDelimiter{\abs}{\lvert}{\rvert}
\DeclarePairedDelimiter{\card}{\lvert}{\rvert}
\DeclarePairedDelimiter{\set}{\lbrace}{\rbrace}
\DeclarePairedDelimiterX{\defset}[2]{\lbrace}{\rbrace}
	{\,#1:#2\,}
\DeclarePairedDelimiter{\fcomp}{\lvert}{\rvert}

\AtBeginDocument{\frenchspacing}

\newcommand{\vertex}{\internalvertex {}}
\newcommand{\internalvertex}{node [circle, inner sep=0pt, outer sep=0pt, minimum size=4pt, fill=black, draw=black]}

\newcommand{\fignotpanAone}[1]{%
\begin{tikzpicture}[y=0.87cm]
\foreach \y in {-0.5,0.5}{
	\draw (0,0) -- (1,\y) (#1+2,\y) -- (#1+3,0);
	\foreach \x in {0,...,#1}{
		\foreach \z in {-0.5,0.5}{
			\draw (\x+1,\y) -- (\x+2,\z);}}}
\foreach \x in {1,...,#1}{
	\draw (\x+1,-0.5) -- (\x+1,0.5);}
\foreach \x in {-1,...,#1}{
	\foreach \y in {-0.5,0.5}{
		\draw (\x+2,\y) \vertex;}}
\draw (0,0) \vertex (#1+3,0) \vertex;
\draw (0,0) node [anchor=south east] {$x$};
\draw (1,0.5) node [anchor=south east] {$y$};
\end{tikzpicture}}

\newcommand{\fignotpanAonethree}[1]{%
\begin{tikzpicture}[y=0.87cm]
\foreach \y in {-0.5,0.5}{
	\draw (0,0) -- (1,\y) (#1+4,\y) -- (#1+5,0);
	\foreach \x in {-2,...,#1}{
		\foreach \z in {-0.5,0.5}{
			\draw (\x+3,\y) -- (\x+4,\z);}}}
\foreach \x in {1,...,#1}{
	\draw (\x+2,-0.5) -- (\x+2,0.5);}
\foreach \x in {-3,...,#1}{
	\foreach \y in {-0.5,0.5}{
		\draw (\x+4,\y) \vertex;}}
\draw (0,0) \vertex (#1+5,0) \vertex;
\draw (0,0) node [anchor=south east] {$x$};
\draw (3,0.5) node [anchor=south east] {$y$};
\end{tikzpicture}}

\newcommand{\figxboxes}{%
\begin{tikzpicture}[baseline]
\foreach \x in {0,...,3}{
	\foreach \y in {0,1}{
		\foreach \z in {0,1}{
			\draw (\x,\y) -- (\x+1,\z);}}}
\foreach \x in {0,...,4}{
	\draw (\x,0) -- (\x,1);
	\foreach \y in {0,1}{
		\draw (\x,\y) \vertex;}}
\draw
	(2,1) node [anchor=south west] {$x$}
	(2,0) node [anchor=north west] {$y$};
\end{tikzpicture}}

\newcommand{\figPthreethreeKtwo}{%
\begin{tikzpicture}[x=0.87cm,baseline]
\foreach \x in {0.5,2.5,4.5}{
	\foreach \y in {0,...,5}{
		\draw (\x,1) -- (\y,0);}}
\draw (0.5,1) -- (4.5,1);
\foreach \y in {0,2,4}
	\draw (\y,0) -- +(1,0);
\foreach \x in {0.5,2.5,4.5}
	\draw (\x,1) \vertex;
\foreach \y in {0,...,5}
	\draw (\y,0) \vertex;
\draw
	(0.5,1) node [anchor=south east] {$x$}
	(4.5,1) node [anchor=south west] {$y$};
\end{tikzpicture}}

\hypersetup{%
	pdfencoding=unicode,
	pdftitle={Some cyclic properties of L₁-graphs},
	pdfauthor={Jonas B. Granholm}}

\begin{document}

\author{Jonas B. Granholm}
\title{Some cyclic properties of $L_1$-graphs}
\date{April 15, 2019}
\maketitle

\begin{abstract}\noindent
A graph $G$ is called an $L_1$-graph if
$d(u)+d(v)\ge\card{N(u)\cup N(v)\cup N(w)}-1$
for every triple of vertices $u,v,w$
where $u$ and~$v$ are at distance~2
and $w\in N(u)\cap N(v)$.
Asratian et~al.~(1996) proved that
all finite connected $L_1$-graphs on at least three vertices such that
$\card{N(u)\cap N(v)}\ge2$ for each pair of vertices
$u,v$ at distance~2
are Hamiltonian,
except for a simple family~$\mathcal{K}$ of exceptions.

We show that not all such graphs are pancyclic,
but that any non-Hamiltonian cycle in such a graph can be extended to a larger cycle
containing all vertices of the original cycle and at most two other vertices.
We also prove a similar result for paths whose endpoints do not have
any common neighbors.
\end{abstract}

\section{Introduction}

We use \cite{bondymurty} for terminology and notation not defined here
and consider simple graphs only.
If $C$ is a cycle in a graph,
then we use the notation~$\dir C$ to denote the cycle with a given direction
and $\revdir C$ for the reverse direction,
and if $x$ is a vertex on the cycle then $x^+$ and $x^-$ denote
the successor and predecessor of~$x$, respectively, in the given direction.
The same notation is used for paths.
A cycle or a path in a finite graph~$G$ is a \emph{Hamilton cycle} or \emph{Hamilton path}, respectively,
if it contains all vertices of~$G$,
and a finite graph is \emph{Hamiltonian} if it contains a Hamilton cycle.
We also use the notation $e(X,Y)$, where $X$ and~$Y$ are vertex sets,
for the number of edges joining a vertex of~$X$ with a vertex of~$Y$.

A~classic result on Hamiltonicity is the following by Dirac~\cite{dirac52}:
A finite graph~$G$ with at least three vertices is Hamiltonian if
$d(v)\ge\card{V(G)}/2$ for every vertex $v\in V(G)$.
This was generalized by Ore~\cite{ore60} as follows:
A finite graph~$G$ with at least three vertices is Hamiltonian if
$d(u)+d(v)\ge\card{V(G)}$ for every pair of non-adjacent vertices $u,v\in V(G)$.
Graphs satisfying this condition are called Ore~graphs,
and there are many results on Hamiltonicity inspired by this theorem.
Nara~\cite{nara80}, among others,
proved that the bound in Ore's theorem can be improved under certain conditions:
\begin{oldtheorem}[see e.g. Nara~\cite{nara80}]
\label{oldthm:nara}
Let $G$ be a finite $2$-connected graph on at least three vertices such that
$d(u)+d(v)\ge\card{V(G)}-1$ for every pair of non-adjacent vertices $u,v\in V(G)$.
Then $G$ is Hamiltonian
unless it belongs to the following set of exceptions:
\[
\mathcal{K}=\defset{G}{K_{p,p+1}\subseteq G\subseteq K_p\join\complement{K_{p+1}} \text{ for some }p\ge2},
\]
where $\join$ denotes the join operation.
\end{oldtheorem}

The above theorems only apply to graphs with large edge density
(\,$|E(G)|\geq \text{constant}\cdot|V(G)|^2$\,)
and diameter~2.
%
Asratian and Khachatryan pioneered a method to overcome this
by using local structures of graphs.
They generalized Ore's theorem to cover sparse graphs with large diameter:
\begin{oldtheorem}[Asratian--Khachatryan~\cite{hasratian90}]
\label{oldthm:L0-hamiltonian}
Let $G$ be a finite connected graph on at least three vertices such that for every triple $u,w,v$ with $d(u,v)=2$ and $w\in N(u)\cap N(v)$ the following property holds:
\[d(u)+d(v)\ge\card{N(u)\cup N(v)\cup N(w)}.\]
Then $G$ is Hamiltonian.
\end{oldtheorem}

A graph is called an $L_i$-graph if
$d(u)+d(v)\ge\card{N(u)\cup N(v)\cup N(w)}-i$
for each triple of vertices $u,v,w$ with $d(u,v)=2$ and $w\in N(u)\cap N(v)$.
Thus \cref{oldthm:L0-hamiltonian} can be reformulated as follows:
all finite connected $L_0$-graphs on at least three vertices are Hamiltonian.

The class of $L_1$-graphs includes not only all $L_0$-graphs and thus all Ore graphs,
but also all claw-free graphs
-- graphs that do not contain $K_{1,3}$ as an induced subgraph~\cite{asratian96c}.
A related result on claw-free graphs is the following by Shi~\cite{shi92}:
Any finite connected claw-free graph on at least three vertices such that
$\card{N(u)\cap N(v)}\ge2$ for each pair of vertices $u,v$ with $d(u,v)=2$ is Hamiltonian.

Every Hamiltonian graph~$G$ is $1$-tough,
that is, it contain no vertex set $S$ such that
the subgraph $G-S$ contains more than $\card{S}$ components.
All $L_0$-graphs and 2-connected claw-free graphs are 1-tough;
for $L_1$-graphs we need a set of exceptions~\cite{asratian96c}:
Any 2-connected $L_1$-graph is either 1-tough
or lies in the set~$\mathcal{K}$ defined above.

In~\cite{asratian96c}, Asratian, Broersma, van~den~Heuvel, and~Veldman
proved the following local analogue of \cref{oldthm:nara},
generalizing \cref{oldthm:L0-hamiltonian}
(note that all $L_0$-graphs satisfy the $\card{N(u)\cap N(v)}\ge2$ condition)
and the result of Shi:
\begin{oldtheorem}[Asratian et~al.~\cite{asratian96c}]
\label{oldthm:L1-hamilton}
Let $G$ be a finite connected $L_1$-graph on at least three vertices
such that $\card{N(u)\cap N(v)}\ge2$ for each pair of vertices $u,v$ with $d(u,v)=2$.
Then $G$ is Hamiltonian unless it belongs to the set~$\mathcal K$.
\end{oldtheorem}
Furthermore, it was proved in~\cite{asratian96c}
that graphs satisfying these conditions have the property that
every pair of vertices at distance at least three is connected by a Hamilton path.

Some other properties of $L_1$-graphs have been found.
Saito~\cite{saito96} showed that all finite 2-connected $L_1$-graphs of diameter~$2$ are Hamiltonian
unless they belong to the set of exceptions~$\mathcal K$,
while Li and Schelp~\cite{li00} showed that every finite 2-connected $L_1$-graph~$G$ with minimum degree
$\delta(G)\ge(\card{V(G)}-2)/3$ is Hamiltonian unless $G\in\mathcal K$.
Furthermore, it was shown in~\cite{asratian96c} that every finite connected $L_1$-graph of even order
has a perfect matching.

\medskip
A finite graph~$G$ is said to be \emph{pancyclic} if it contains a cycle of each length
from~$3$ up to~$\card{V(G)}$.
Bondy~\cite{bondy71} proved that all Ore~graphs are pancyclic,
except for complete bipartite graphs $K_{n,n}$, $n\ge2$.
He also made a metaconjecture that
almost any nontrivial condition that implies Hamiltonicity also implies pancyclicity,
though there may be a simple family of exceptional graphs.
Aldred, Holton, and Min~\cite{aldred94} proved that
graphs satisfying the conditions of \cref{oldthm:nara} are pancyclic,
except for the graphs in the set~$\mathcal{K}$,
complete bipartite graphs $K_{n,n}$, and the cycle $C_5$.

An even stronger property is called \emph{cycle extendability},
which means that any cycle that does not include all vertices of the graph
can be extended to a new cycle containing
a single new vertex in addition to all vertices of the original cycle.
This notion was introduced by Hendry~\cite{hendry90}, who also proved that Ore~graphs,
with a relatively complicated set of exceptions, are cycle extendable.
Without any exceptions, however, Bondy~\cite{bondy81} had earlier proved that any cycle in an Ore~graph
that does not include all vertices
can be extended to a larger cycle containing all vertices of the original cycle and at most two other vertices.

$L_0$-graphs (with the exception of the graphs $K_{n,n}$)
have also been found to be pancyclic by Asratian and Sarkisian~\cite{asratian94cite3}.
They further proved the following:
\begin{oldtheorem}[Asratian--Sarkisian~\cite{asratian96b}]
\label{oldthm:asratian96b}
Let $G$ be a finite connected $L_0$-graph on at least three vertices.
Then for each $\ell=4,\dotsc,\card{V(G)}$,
unless $G=K_{n,n}$ for some $n\ge2$,
every vertex of~$G$ lies on a cycle of length~$\ell$,
every edge of~$G$ that does not lie on a triangle lies on a cycle of length~$\ell$,
and every pair of vertices at distance no less than three and at most $\ell$
is connected by a path with $\ell$~vertices.
\end{oldtheorem}

\medskip
In 2004, Diestel and Kühn~\cite{diestel04a} suggested a new concept for
infinite locally finite graphs (infinite graphs with only finite vertex degrees),
called \emph{Hamilton circles},
which are analogues of Hamilton cycles in finite graphs.
Let $G$ be an infinite locally finite graph.
A \emph{ray} in~$G$ is a one-way infinite path.
We define an equivalence relation on the set of rays in~$G$
by saying that two rays are equivalent if
they have a subray in the same component of $G-S$ for every finite vertex set~$S$.
The equivalence classes of this relation are called the \emph{ends} of~$G$,
and can be seen as points at infinity.
The \emph{Freudenthal compactification}~$\fcomp{G}$ of~$G$
is a topological space constructed by viewing $G$ as a 1-complex,
and adding the ends of~$G$ as additional points.
Finally, a Hamilton circle in the Freudenthal compactification~$\fcomp{G}$
is a homeomorphic image of the unit circle
that passes through every vertex and every end exactly once.
For a more thorough exposition, see~\cite{diestel09}.

Diestel~\cite{diestel10} launched the ambitious project of
extending results on finite Hamilton cycles to Hamilton circles.
Georgakopoulos~\cite{georgakopoulos09} showed that
if $G$ is the square of a 2-connected, infinite, locally finite graph,
then $\fcomp{G}$ has a Hamilton circle,
extending Fleischner's theorem~\cite{fleischner74} for finite graphs.
Heuer~\cite{heuer15} and Hamann et~al.~\cite{hamann16} showed that
the Freudenthal compactification of every connected, locally connected, infinite, locally finite, claw-free graph has a Hamilton circle,
extending Oberly--Sumner's theorem~\cite{oberly79}.

Heuer~\cite{heuer16} furthermore proved that
the Freudenthal compactification of every claw-free, locally connected graph
satisfying the conditions of \cref{oldthm:L0-hamiltonian}
has a Hamilton circle.
It is easy to see that
for a triple $u,w,v$ with $d(u,v)=2$ and $w\in N(u)\cap N(v)$ in a claw-free graph,
the inequality $d(u)+d(v)\ge\card{N(u)\cup N(v)\cup N(w)}$
is equivalent to the inequality $\card{N(u)\cap N(v)}\ge2$.
Thus the result of Heuer~\cite{heuer16} can be reformulated as follows:
\begin{oldtheorem}
\label{oldthm:claw-free-ge2-circle}
Let $G$ be a locally finite, connected, claw-free graph on at least three vertices
such that $\card{N(u)\cap N(v)}\ge2$ for each pair of vertices $u,v$ with $d(u,v)=2$.
Then $\fcomp{G}$ has a Hamilton circle.
\end{oldtheorem}

Kündgen, Li, and Thomassen~\cite{kundgen17} introduced
another concept for infinite locally finite graphs:
A closed curve in the Freudenthal compactification~$\fcomp{G}$ is called
a \emph{Hamilton curve} if it meets every vertex exactly once,
but is allowed to meet the ends of~$\fcomp{G}$ multiple times.
They showed that the condition of \cref{oldthm:L0-hamiltonian}
implies the existence of a Hamilton curve.

\medskip
In this article,
which is partly based on the author's master's thesis~\cite{granholm16},
we investigate $L_1$-graphs in the same spirit as~\cref{oldthm:asratian96b},
and show that they, unlike $L_0$-graphs, need not be pancyclic.
However, we prove that if $G$ is a locally finite graph (not necessarily finite)
satisfying the conditions of~\cref{oldthm:L1-hamilton}, then
\begin{itemize}
\item any cycle~$C$ in $G$ that does not contain all vertices of~$G$
can be extended to a larger cycle
containing all vertices of~$C$ and at most two other vertices;
\item for any pair of vertices $x,y$ with no common neighbors
and any $\pathto xy$-path~$P$ in $G$ that does not include all vertices of~$G$,
there is a longer $\pathto xy$-path containing
all vertices of~$P$ and at most two other vertices.
\end{itemize}
Furthermore we show that if $G$ is an infinite, locally finite graph
satisfying the conditions of~\cref{oldthm:L1-hamilton},
then $\fcomp{G}$ has a Hamilton curve.
Finally, we provide a characterization of all connected bipartite $L_1$-graphs.

\section{Results}

The main result of this paper is the following theorem:
\begin{theorem}
\label{thm:cycleextension}
Let $G$ be a connected, locally finite $L_1$-graph on at least three vertices such that
$\card{N(u)\cap N(v)}\ge2$ for each pair of vertices $u,v$ with $d(u,v)=2$.
Then for every cycle~$C_n$ of length~$n$ in~$G$
that does not contain all vertices of~$G$,
there is a cycle~$C_{n+\ell}$ of length~$n+\ell$, where~$1\le\ell\le2$,
such that $V(C_n)\subset V(C_{n+\ell})$,
unless $n=\card{V(G)}-1$ and $G\in\mathcal{K}$.
\end{theorem}

Unlike for \cref{oldthm:nara},
graphs satisfying the conditions of \cref{oldthm:L1-hamilton} need not be pancyclic,
so \cref{thm:cycleextension} is best possible.
The graph
$K_1\seqjoin\complement{K_2}\seqjoin K_2\seqjoin K_2\seqjoin\complement{K_2}\seqjoin K_1$
(see \cref{fig:K1-E2-K2-K2-E2-K1}), for example,
has 10~vertices and does not contain a 9-cycle.
In general, the graph
\[G=K_1\seqjoin\complement{K_2}\seqjoin
{\underbrace{K_2\seqjoin\dotsb\seqjoin K_2}_{\mathclap{\text{at least two copies of }K_2}}}
\seqjoin\complement{K_2}\seqjoin K_1\]
does not contain any cycle of length $\card{V(G)}-1$.
Furthermore, the graph
in \cref{fig:K1-E2-E2-K2-K2-E2-E2-K1}
has 14~vertices and does not contain any cycle of length 11~or~13,
and can be extended to an infinite family of graphs in the same way as above.
\begin{figure}
\centering
\fignotpanAone{2}
\caption{The graph
$K_1\seqjoin\complement{K_2}\seqjoin K_2\seqjoin K_2\seqjoin\complement{K_2}\seqjoin K_1$}
\label{fig:K1-E2-K2-K2-E2-K1}
\end{figure}
\begin{figure}
\centering
\fignotpanAonethree{2}
\caption{The graph
$K_1\seqjoin\complement{K_2}\seqjoin\complement{K_2}\seqjoin K_2
\seqjoin K_2\seqjoin\complement{K_2}\seqjoin\complement{K_2}\seqjoin K_1$}
\label{fig:K1-E2-E2-K2-K2-E2-E2-K1}
\end{figure}

It is easy to see that every vertex in a graph
satisfying the conditions of \cref{thm:cycleextension}
lies on a cycle of length at most~$4$.
Thus we can draw the following conclusions:
\begin{corollary}
Let $G$ be a finite connected $L_1$-graph on at least three vertices such that
$\card{N(u)\cap N(v)}\ge2$ for each pair of vertices $u,v$ with $d(u,v)=2$.
Then for each vertex $x\in V(G)$ there is
a number $r$ and
a sequence of integers $n_1,n_2,\dotsc,n_r$,
depending on $x$,
such that
$n_1\le4$,
$n_r=\card{V(G)}$ (unless $G\in\mathcal K$, in which case $n_r=\card{V(G)}-1$),
and
$1\le n_{i+1}-n_i\le2$ for each $i=1,\dotsc,r-1$,
and a sequence of cycles $C_{n_1},C_{n_2},\dotsc,C_{n_r}$
of lengths $n_1,n_2,\dotsc,n_r$ respectively,
such that
$x\in V(C_{n_1})\subset V(C_{n_2})\subset \dotsb\subset V(C_{n_r})$.
\end{corollary}

\begin{corollary}
Let $G$ be a finite connected $L_1$-graph on at least three vertices such that
$\card{N(u)\cap N(v)}\ge2$ for each pair of vertices $u,v$ with $d(u,v)=2$.
Then for each vertex~$x\in V(G)$ and each $m=4,\dotsc,\card{V(G)}$,
the vertex $x$ lies on a cycle of length $m$ or $m-1$.
\end{corollary}

Using the same reasoning we also get the following:
\begin{corollary}
Let $G$ be a connected, infinite, locally finite $L_1$-graph on at least three vertices such that
$\card{N(u)\cap N(v)}\ge2$ for each pair of vertices $u,v$ with $d(u,v)=2$.
Then for each vertex~$x\in V(G)$ and each $m\ge4$,
the vertex $x$ lies on a cycle of length $m$ or $m-1$.
\end{corollary}

\medskip
We will also prove the following theorems:
\begin{theorem}
\label{thm:pathextension-adjacent}
Let $G$ be a connected, locally finite $L_1$-graph on at least three vertices such that
$\card{N(u)\cap N(v)}\ge2$ for each pair of vertices $u,v$ with $d(u,v)=2$,
and let $x$~and~$y$ be two adjacent vertices in~$G$ with no neighbors in common.
Then for every $\pathto xy$-path~$P_n$ with $n$~vertices in~$G$
that does not contain all vertices of~$G$,
there is an $\pathto xy$-path~$P_{n+\ell}$ with $n+\ell$~vertices, $1\le\ell\le2$,
such that $V(P_n)\subset V(P_{n+\ell})$,
unless $n=\card{V(G)}-1$ and $G\in\mathcal{K}$.
\end{theorem}

\begin{theorem}
\label{thm:pathextension-ge3}
Let $G$ be a connected, locally finite $L_1$-graph on at least three vertices such that
$\card{N(u)\cap N(v)}\ge2$ for each pair of vertices $u,v$ with $d(u,v)=2$,
and let $x$~and~$y$ be two vertices in~$G$ with $d(x,y)\ge3$.
Then for every $\pathto xy$-path~$P_n$ with $n$~vertices in~$G$
that does not contain all vertices of~$G$,
there is an $\pathto xy$-path~$P_{n+\ell}$ with $n+\ell$~vertices, $1\le\ell\le2$,
such that $V(P_n)\subset V(P_{n+\ell})$.
\end{theorem}

\Cref{thm:pathextension-adjacent,thm:pathextension-ge3}
can be stated together as a single result
by removing the requirement that $x$ and~$y$ are adjacent
from the formulation of \cref{thm:pathextension-adjacent},
that is, $x$ and~$y$ can be any pair of vertices without common neighbors.

The results in \cref{thm:pathextension-adjacent,thm:pathextension-ge3} are sharp;
in the graph in \cref{fig:K1-E2-K2-K2-E2-K1} there are no $\pathto xy$-paths with 9 vertices,
and in the graph in \cref{fig:K1-E2-E2-K2-K2-E2-E2-K1}
there are no $\pathto xy$-paths with 11 or 13 vertices.
%
%
Furthermore, the results
cannot simply be extended to cover the case when $x$ and $y$ have neighbors in common;
some counterexamples can be seen in \cref{fig:commonneighbors}.
\begin{figure}
\centering
\figxboxes
\qquad
\figPthreethreeKtwo
\caption{The graphs $K_2\join K_2\join K_2\join K_2\join K_2$ and $P_3\join 3K_2$}
\label{fig:commonneighbors}
\end{figure}

\begin{corollary}
Let $G$ be a finite connected $L_1$-graph on at least three vertices such that
$\card{N(u)\cap N(v)}\ge2$ for each pair of vertices $u,v$ with $d(u,v)=2$.
Then for every pair of vertices $x,y\in V(G)$ with no neighbors in common, there is
a number~$r$ and
a sequence of integers $n_1,n_2,\dotsc,n_r$,
depending on $x$~and~$y$,
such that
$n_1=d(x,y)+1$,
$n_r=\card{V(G)}$ (unless $G\in\mathcal K$, in which case $n_r=\card{V(G)}-1$),
and
$1\le n_{i+1}-n_i\le2$ for each $i=1,\dotsc,r-1$,
and a sequence of $\pathto{x}{y}$-paths $P_{n_1},P_{n_2},\dotsc,P_{n_r}$
with $n_1,n_2,\dotsc,n_r$ vertices, respectively,
such that
$V(P_{n_1})\subset V(P_{n_2})\subset \dotsb\subset V(P_{n_r})$.
\end{corollary}

\begin{corollary}
Let $G$ be a finite connected $L_1$-graph on at least three vertices such that
$\card{N(u)\cap N(v)}\ge2$ for each pair of vertices $u,v$ with $d(u,v)=2$.
Then for every pair of vertices $x,y\in V(G)$ with no neighbors in common
and each $m=d(x,y)+1,\dotsc,\card{V(G)}$,
there is an $\pathto{x}{y}$-path with $m$ or $m-1$ vertices.
\end{corollary}

\begin{corollary}
Let $G$ be a connected, infinite, locally finite $L_1$-graph on at least three vertices such that
$\card{N(u)\cap N(v)}\ge2$ for each pair of vertices $u,v$ with $d(u,v)=2$.
Then for every pair of vertices $x,y\in V(G)$ with no neighbors in common
and each $m\ge d(x,y)+1$,
there is an $\pathto{x}{y}$-path with $m$ or $m-1$ vertices.
\end{corollary}


\medskip
The local nature of the $L_1$-condition allows us to easily extend \cref{oldthm:L1-hamilton}
to Hamilton curves in infinite graphs.
\begin{theorem}
\label{thm:infinite}
Let $G$ be a connected, infinite, locally finite $L_1$-graph on at least three vertices such that
$\card{N(u)\cap N(v)}\ge2$ for each pair of vertices $u,v$ with $d(u,v)=2$.
Then $\fcomp{G}$ has a Hamilton curve.
\end{theorem}
We believe that \cref{thm:infinite} can be strengthened to the following,
which would be a generalization of \cref{oldthm:claw-free-ge2-circle}:
\begin{conjecture}
\label{conj:circle}
Let $G$ be a connected, infinite, locally finite $L_1$-graph on at least three vertices
such that $\card{N(u)\cap N(v)}\ge2$ for each pair of vertices $u,v$ with $d(u,v)=2$.
Then $\fcomp{G}$ has a Hamilton circle.
\end{conjecture}


\medskip
We end by characterizing all bipartite $L_1$-graphs.

\begin{theorem}
\label{thm:bipartite}
Let $G$ be a connected, bipartite $L_1$-graph
with maximum degree greater than~$2$.
Then either $G$ is a complete bipartite graph~$K_{n,n}$,
or $G$ is obtained from~$K_{n,n}$
by removing a single vertex, edge, or perfect matching.
\end{theorem}
Note that a connected bipartite $L_1$-graph with maximum degree at most~$2$ is
either an even cycle or a finite or infinite path.

\section{Proofs}
In this section we prove our results.

\begin{remark}
\label{prop:unionreformulation}
Let $uwv$ be a path in $G$ with $uv\notin E(G)$. Then the inequality
$d(u)+d(v)\ge\card{N(u)\cup N(v)\cup N(w)}-1$
is equivalent to
$\card{N(u)\cap N(v)}\ge\card[\big]{N(w)\setminus\bigl(N(u)\cup N(v)\bigr)}-1$.
\end{remark}

\begin{lemma}
\label{lem:ge22connected}
If $G$ is a connected graph with at least three vertices such that
$\card{N(u)\cap N(v)}\ge2$ for each pair of vertices $u,v$ with $d(u,v)=2$,
then $G$ is $2$-connected.
\end{lemma}

\begin{lemma}[see {\cite[thm.~5]{asratian96c}}%
\footnote{In~\cite{asratian96c}, the result in \cref{thm:1-tough-or-K}
is only stated for finite graphs,
but the same proof works for infinite, locally finite graphs as well.}]
\label{thm:1-tough-or-K}
If $G$ is a $2$-connected $L_1$-graph, then either $G$ is $1$-tough or $G\in\mathcal K$.
\end{lemma}

\subsection{Proof of \texorpdfstring{\cref{thm:cycleextension}}{Theorem~1}}

Assume that there is no cycle of length $n+1$ or $n+2$ containing the vertices of~$C_n$.
Specify a cyclic orientation~$\dir C_n$ of~$C_n$ and pick a vertex $v\in V(G)\setminus V(C_n)$ such that $N(v)\cap V(C_n)\ne\emptyset$.
Set $W=N(v)\cap V(C_n)$ and $p=\card{W}$. Let $w_1,\dotsc,w_p$ be the vertices of $W,$ occurring on $\dir C_n$ in the order of their indices, and set $W^+=\set{w_1^+,\dotsc,w_p^+}$.
All indices are considered to be modulo $p$, so $w_{p+1}=w_1$.

\begin{remark}
\label{rem:extension}
Note that any extension of $C_n$ that occurs in this proof
contains either the vertex $v$ or a vertex of $M_3(v)$
(in \cref{thm:unionge2cycleextension:claim:N(wi),thm:unionge2cycleextension:claim:everysecond}
it will always be the case that $v$ is included).
This will be important in the proof of \cref{thm:infinite}.
\end{remark}

\begin{claim}
\label{thm:unionge2cycleextension:claim:N(wi)}
The set $W^+\cup\set v$ is independent,
$N(w_i^+)\cap N(v)=N(w_i^+)\cap W,$
$\card{N(w_i)\cap W^+}=\card{N(w_i^+)\cap W}$, and
$N(w_i)\setminus\bigl(N(w_i^+)\cup N(v)\cup\set v\bigr)\subseteq W^+$
for $i=1,\dotsc,p$.
\end{claim}
\begin{proof}
If there is an edge $vw_i^+\in E(G)$, then $G$ contains an $(n+1)$-cycle $w_ivw_i^+\dir C_nw_i$,
and if there is an edge $w_i^+w_j^+\in E(G)$, then $G$ contains an $(n+1)$-cycle $w_ivw_j\revdir C_nw_i^+w_j^+\dir C_nw_i$.
Thus
\begin{equation}
\label{thm:unionge2cycleextension:eq:independent}
W^+\cup\set v\text{ is an independent set.}
\end{equation}
Also, if $\bigl(N(w_i^+)\cap N(v)\bigr)\setminus V(C_n)\ne\emptyset$ for some $1\le i\le p$,
that is, if $w_i^+$ and $v$ have a common neighbor $u$ outside $C_n$,
then $G$ contains an $(n+2)$-cycle \smash{$w_ivuw_i^+\dir C_nw_i$}.
Thus $\bigl(N(w_i^+)\cap N(v)\bigr)\setminus V(C_n)=\emptyset$, which means that
\begin{equation}
\label{thm:unionge2cycleextension:eq:N(wi+)W}
N(w_i^+)\cap N(v)=N(w_i^+)\cap W.
\end{equation}

Now for each $i=1,\dotsc,p$, we have $d(v,w_i^+)=2$ and $w_i\in N(w_i^+)\cap N(v)$, so by \cref{prop:unionreformulation},
\begin{equation}
\label{thm:unionge2cycleextension:eq:transformedunionwi}
\card{N(w_i^+)\cap W}=\card{N(w_i^+)\cap N(v)}\ge\card[\big]{N(w_i)\setminus\bigl(N(w_i^+)\cup N(v)\bigr)}-1.
\end{equation}
Obviously,
\begin{equation}
\label{thm:unionge2cycleextension:eq:N(wi)-inclusion}
N(w_i)\cap W^+\subseteq N(w_i)\setminus\bigl(N(w_i^+)\cup N(v)\cup\set v\bigr).
\end{equation}
Thus $\card{N(w_i)\cap W^+}\le \card[\big]{N(w_i)\setminus\bigl(N(w_i^+)\cup N(v)\bigr)}-1$.
This and \eqref{thm:unionge2cycleextension:eq:transformedunionwi} together imply that
\begin{equation}
\label{thm:unionge2cycleextension:eq:WW+inequality}
\card{N(w_i)\cap W^+}\le\card{N(w_i^+)\cap W}.
\end{equation}
We will now count the number of edges between $W^+$ and $W$ in two different ways:
\begin{equation}
e(W^+,W)=\sum_{i=1}^p\card{N(w_i)\cap W^+}\le\sum_{i=1}^p\card{N(w_i^+)\cap W}=e(W^+,W).
\end{equation}
It follows for each $i=1,\dotsc,p$, that
\begin{equation}
\label{thm:unionge2cycleextension:eq:WW+equality}
\card{N(w_i)\cap W^+}=\card{N(w_i^+)\cap W}
\end{equation}
and that we have equality in \eqref{thm:unionge2cycleextension:eq:N(wi)-inclusion}, so
\begin{equation}
\label{thm:unionge2cycleextension:eq:N(wi)}
N(w_i)\setminus\bigl(N(w_i^+)\cup N(v)\cup\set v\bigr)=N(w_i)\cap W^+\subseteq W^+.
\qedhere
\end{equation}
\end{proof}

\begin{claim}
\label{thm:unionge2cycleextension:claim:everysecond}
$w_i^+=w_{i+1}^-$ for $i=1,\dotsc,p$,
that is, $n=2p$ and
$v$ is adjacent to every second vertex of $C_n$.
\end{claim}
\begin{proof}
Suppose that $v$ is not adjacent to every second vertex of the cycle $C_n$. Then $w_i^+\ne w_{i+1}^-$ for some $i$.
Without loss of generality, assume that $w_1^+\ne w_2^-$, which means that $w_2^-\notin W^+$.
This and \eqref{thm:unionge2cycleextension:eq:N(wi)} for $i=2$ imply that $w_2^-\in N(w_2^+)$, because otherwise $w_2^-\in N(w_2)\setminus\bigl(N(w_2^+)\cup N(v)\cup\set v\bigr)\subseteq W^+$, a contradiction.
Therefore $w_2^-w_2^+\in E(G)$.
This in turn means that $w_2^+\ne w_3^-$, because otherwise there would be an $(n+1)$-cycle $w_2^-w_2^+w_2vw_3\dir C_nw_2^-$
(unless $p=1$, in which case recall that $w_{p+1}=w_1$ and skip this sentence).
Repetition of this argument shows that $w_i^+\ne w_{i+1}^-$ for $i=1,\dotsc,p$, and that
\begin{equation}
\label{thm:unionge2cycleextension:eq:edgeunder}
w_i^+w_i^-\in E(G)\text{ for each }i=1,\dotsc,p.
\end{equation}

Now it is easy to see that $w_1^+w_j\notin E(G)$ for each $j\ne 1$, as otherwise there would be an $(n+1)$-cycle $w_1vw_jw_1^+\dir C_nw_j^-w_j^+\dir C_nw_1$ containing the vertices of $C_n$.
This, together with \eqref{thm:unionge2cycleextension:eq:N(wi+)W}, implies that $N(w_1^+)\cap N(v)=\set{w_1}$. This contradicts the fact that $d(w_1^+,v)=2$.
Thus we can conclude that $w_i^+=w_{i+1}^-$ for each $i=1,\dotsc,p$, and that $n=2p$.
\end{proof}

\begin{claim}
\label{thm:unionge2cycleextension:claim:GinK}
$n=\card{V(G)}-1$ and $G\in\mathcal{K}$.
\end{claim}
\begin{proof}
We have concluded that $n=2p$ and that $N(v)$ contains every second vertex of~$C_n$.
Note that $p\ge2$,
as otherwise $N(w_1^+)\cap N(v)=\set{w_1}$ by \cref{thm:unionge2cycleextension:claim:N(wi)},
contradicting the conditions of the theorem.
Suppose some vertex $w_i^+\in W^+$ has a neighbor $u$ outside $C_n$.
Since $v$ was picked arbitrarily in the set $V(G)\setminus V(C_n)$ such that $N(v)\cap V(C_n)\ne\emptyset$,
we can conclude that $u$ is adjacent to every second vertex of $C_n$ as well,
that is, $N(u)\cap V(C_n)=W^+$.
But then there is an $(n+2)$-cycle $w_1vw_2w_1^+uw_2^+\dir C_nw_1$ containing the vertices of $C_n$, a contradiction,
so no vertex outside $C_n$ is adjacent to any vertex in~$W^+$.
Thus $G$ is not 1-tough, so $G\in\mathcal K$ by \cref{thm:1-tough-or-K}.
Also, since $G\in\mathcal K$ it follows that
if $n<\card{V(G)}-1$ then there is
a cycle of length $n+1$ or $n+2$ in~$G$ containing the vertices of~$C_n$.
Thus $n=\card{V(G)}-1$.
\end{proof}

\subsection{Proof of \texorpdfstring{\cref{thm:pathextension-adjacent}}{Theorem~5}}

Assume that there is no $\pathto xy$-path with $n+1$ or $n+2$ vertices containing the vertices of $P_n$.
Pick a vertex $v\in V(G)\setminus V(P_n)$ such that $N(v)\cap V(P_n)\ne\emptyset$.
Since $x$ and $y$ have no neighbors in common,
it follows that $\card{N(v)\cap\set{x,y}}\le1$.
Without loss of generality we assume that $vy\notin E(G)$.
Let $\dir P_n$ be $P_n$ directed from $x$ to $y$.
Set $W=N(v)\cap V(P_n)$ and $p=\card{W}$. Let $w_1,\dotsc,w_p$ be the vertices of $W,$ occurring on $\dir P_n$ in the order of their indices, and set $W^+=\set{w_1^+,\dotsc,w_p^+}$.
The path $P_n$ together with the edge $xy$ of course forms a cycle, and for simplicity we define $z^+$ to be the successor of $z$ on this cycle, so $y^+=x$, etc.
Also, all indices are considered to be modulo $p$, so $w_{p+1}=w_1$.

\begin{claim}
\label{thm:unionge2pathextension:claim:N(wi)}
The set $W^+\cup\set v$ is independent,
$N(w_i^+)\cap N(v)=N(w_i^+)\cap W,$
$\card{N(w_i)\cap W^+}=\card{N(w_i^+)\cap W}$, and
$N(w_i)\setminus\bigl(N(w_i^+)\cup N(v)\cup\set v\bigr)\subseteq W^+$
for $i=1,\dotsc,p$.
\end{claim}
\begin{proof}
This follows using the same arguments as in the proof of \cref{thm:cycleextension}.
\end{proof}

\begin{claim}
\label{thm:unionge2pathextension:claim:everysecond}
$w_1=x$, $w_p=y^-$, and $w_i^+=w_{i+1}^-$ for $i=1,\dotsc,p-1$,
that is, $n=2p$ and
$v$ is adjacent to every second vertex of $P_n$.
\end{claim}
\begin{proof}
We will start by showing that $w_i^+=w_{i+1}^-$ for each $i=1,\dotsc,p-1$.
Assume on the contrary that $w_k^+\ne w_{k+1}^-$ for some $k\le p-1$,
and furthermore assume that $k$ is the first such index, i.e., either $k=1$ or $k\ge2$ and $w_i^+=w_{i+1}^-$ for every $i=1,\dotsc,k-1$.
Then $w_{k+1}^-\notin W^+$.
This and \eqref{thm:unionge2cycleextension:eq:N(wi)} for $i=k+1$ imply that $w_{k+1}^-\in N(w_{k+1}^+)$, because otherwise $w_{k+1}^-\in N(w_{k+1})\setminus\bigl(N(w_{k+1}^+)\cup N(v)\cup\set v\bigr)\subseteq W^+$, a contradiction.
Therefore $w_{k+1}^-w_{k+1}^+\in E(G)$.
This in turn means that $w_{k+1}^+\ne w_{k+2}^-$, because otherwise there would be an $\pathto xy$-path $x\dir P_nw_{k+1}^-w_{k+1}^+w_{k+1}vw_{k+2}\dir P_ny$
with $n+1$ vertices
(unless $k=p-1$, in which case skip this sentence).
Repetition of this argument shows that $w_i^+\ne w_{i+1}^-$ for each $i=k,\dotsc,p-1$, and that
\begin{equation}
\label{thm:unionge2pathextension:eq:edgeunder}
w_i^+w_i^-\in E(G)\text{ for each }i=k+1,\dotsc,p.
\end{equation}

Let $W_1=\set{w_1,\dotsc,w_k}$ and $W_1^+=\set{w_1^+,\dotsc,w_k^+}$.
It is easy to see that $w_i^+w_j\notin E(G)$ for each $j>k$ and each $i\ne j$, as otherwise there would be an $\pathto xy$-path $x\dir P_nw_ivw_jw_i^+\dir P_nw_j^-w_j^+\dir P_ny$ (if $i<j$) or $x\dir P_nw_j^-w_j^+\dir P_nw_ivw_jw_i^+\dir P_ny$ (if $i>j$) with $n+1$ vertices.
This means that $N(w_i^+)\cap W=N(w_i^+)\cap W_1$ for each $i=1,\dotsc,k$. This, together with
\eqref{thm:unionge2cycleextension:eq:WW+equality}, means that
\begin{equation}
\label{thm:unionge2pathextension:eq:NcapWinequality}
\card{N(w_i)\cap W_1^+}\le
\card{N(w_i)\cap W^+}=
\card{N(w_i^+)\cap W}=
\card{N(w_i^+)\cap W_1}
\end{equation}
for every $i=1,\dotsc,k$.
We will now count the edges between $W_1^+$ and $W_1$ in two different ways:
\begin{equation}
e(W_1^+,W_1)=\sum_{i=1}^k\card{N(w_i)\cap W_1^+}\le
\sum_{i=1}^k\card{N(w_i^+)\cap W_1}=e(W_1^+,W_1).
\end{equation}
This means that we have equality in \eqref{thm:unionge2pathextension:eq:NcapWinequality}, so for every $i=1,\dotsc,k$
\begin{equation}
\card{N(w_i)\cap W_1^+}=
\card{N(w_i)\cap W^+},
\end{equation}
which means that $w_iw_j^+\notin E(G)$ for all $i=1,\dotsc,k$ and $j=k+1,\dotsc,p$.
But then \eqref{thm:unionge2cycleextension:eq:N(wi+)W} implies that $N(v)\cap N(w_j^+)=\set{w_j}$ for every $j=k+1,\dotsc,p$.
This contradicts the assumptions of the theorem, because the fact that $d(v,w_j^+)=2$ implies that $\card{N(v)\cap N(w_j^+)}\ge2$.
Thus we can conclude that $w_i^+=w_{i+1}^-$ for each $i=1,\dotsc,p-1$.

Now we can use an argument similar to the one in the beginning of this proof to show that $w_1=x$:
If $w_1\ne x$ then $w_1^-\notin W^+$ (by assumption $vy\notin E(G)$, so no vertex on $x\dir P_nw_1$ is in $W^+$).
This means that $w_1^-\in N(w_1^+)$ by \eqref{thm:unionge2cycleextension:eq:N(wi)}, so $w_1^+w_1^-\in E(G)$.
Note also that $p\ge2$, since otherwise $N(v)\cap N(w_1^+)=\set{w_1}$, a contradiction as $d(v,w_1^+)=2$.
But now, since $w_1^+=w_2^-$, there is an $\pathto xy$-path $x\dir P_nw_1^-w_1^+w_1vw_2\dir P_ny$ with $n+1$ vertices.
This is a contradiction, so we can conclude that $w_1=x$.
Also, since $y$ and $x=w_1$ are adjacent and have no neighbors in common and $y\ne v$, it follows that $y\in N(w_1)\setminus\bigl(N(w_1^+)\cup N(v)\cup\set v\bigr)$.
Thus $y\in W^+$ by \eqref{thm:unionge2cycleextension:eq:N(wi)}, so $w_p=y^-$, and $n=2p$.
\end{proof}

\begin{claim}
\label{thm:unionge2pathextension:claim:GinK}
$n=\card{V(G)}-1$ and $G\in\mathcal{K}$.
\end{claim}
\begin{proof}
This follows using the same arguments as in the proof of \cref{thm:cycleextension}.
\end{proof}

\subsection{Proof of \texorpdfstring{\cref{thm:pathextension-ge3}}{Theorem~6}}

Assume that there is no $\pathto xy$-path with $n+1$ or $n+2$ vertices containing the vertices of $P_n$.
Pick a vertex $v\in V(G)\setminus V(P_n)$ such that $N(v)\cap V(P_n)\ne\emptyset$.
Since $d(x,y)\ge3$, it follows that $\card{N(v)\cap\set{x,y}}\le1$.
Without loss of generality we assume that $vy\notin E(G)$.
Let $\dir P_n$ be $P_n$ directed from $x$ to~$y$.
Set $W=N(v)\cap V(P_n)$ and $p=\card{W}$. Let $w_1,\dotsc,w_p$ be the vertices of~$W$,
occurring on $\dir P_n$ in the order of their indices, and set $W^+=\set{w_1^+,\dotsc,w_p^+}$.

\begin{claim}
\label{thm:ge3:claim:N(wi)}
The set $W^+\cup\set v$ is independent,
$N(w_i^+)\cap N(v)=N(w_i^+)\cap W,$
$\card{N(w_i)\cap W^+}=\card{N(w_i^+)\cap W}$, and
$N(w_i)\setminus\bigl(N(w_i^+)\cup N(v)\cup\set v\bigr)\subseteq W^+$
for $i=1,\dotsc,p$.
\end{claim}
\begin{proof}
This is proved exactly as \cref{thm:unionge2pathextension:claim:N(wi)}
in the proof of \cref{thm:pathextension-adjacent}.
\end{proof}

\begin{claim}
\label{thm:ge3:claim:everysecond}
$w_1=x$ and $w_i^+=w_{i+1}^-$ for $i=1,\dotsc,p-1$,
that is, $v$ is adjacent to every second vertex of $x\dir P_nw_p$.
\end{claim}
\begin{proof}
This is proved exactly as \cref{thm:unionge2pathextension:claim:everysecond}
in the proof of \cref{thm:pathextension-adjacent},
without the last two sentences.
\end{proof}

\begin{claim}
\label{thm:ge3:claim:t}
There exists a number~$t$ such that
$\card{N(w_i)\cap W^+}=\card{N(w_i^+)\cap W}=t$ for $i=1,\dotsc,p$.
\end{claim}
\begin{proof}
First, for $i=1,\dotsc,p-1$
\begin{equation}
\begin{aligned}
\label{eq:decreasing-t}
\card{N(w_i^+)\cap W}
&=\card{N(w_i^+)\cap N(v)}\\
&\ge\card[\big]{N(w_{i+1})\setminus\bigl(N(w_i^+)\cup N(v)\bigr)}-1\\
&\ge\card{N(w_{i+1})\cap(W^+\cup\set v)}-1\\
&=\card{N(w_{i+1}^+)\cap W}.
\end{aligned}
\end{equation}
Now for any $k\in\set{1,\dotsc,p-1}$, $w_k^+w_{k+1}\in E(G)$, so
\begin{equation}
\begin{aligned}
1+e\bigl(\set{w_{k+1}^+,\dotsc,w_p^+},\set{w_{k+1},\dotsc,w_p}\bigr)
&\le e\bigl(W^+,\set{w_{k+1},\dotsc,w_p}\bigr)\\
&=\sum_{i=k+1}^p\card{N(w_i)\cap W^+}\\
&=\sum_{i=k+1}^p\card{N(w_i^+)\cap W}\\
&=e\bigl(\set{w_{k+1}^+,\dotsc,w_p^+},W\bigr).
\end{aligned}
\end{equation}
Thus $G$ must contain some edge $w_j^+w_i$ with $i\le k<j$.
Now, by using \eqref{eq:decreasing-t} iteratively,
\begin{equation}
\begin{aligned}
\card{N(w_{k+1}^+)\cap W}
&\ge\card{N(w_j^+)\cap W}\\
&=\card{N(w_j^+)\cap N(v)}\\
&\ge\card[\big]{N(w_i)\setminus\bigl(N(w_j^+)\cup N(v)\bigr)}-1\\
&\ge\card{N(w_i)\cap W^+}\\
&=\card{N(w_i^+)\cap W}\\
&\ge\card{N(w_k^+)\cap W}.
\end{aligned}
\end{equation}
We can thus conclude that
$\card{N(w_k^+)\cap W}=\card{N(w_{k+1}^+)\cap W}$ for $k=1,\dotsc,p-1$.
The rest of the claim now follows from \cref{thm:ge3:claim:N(wi)}.
\end{proof}

\begin{claim}
\label{thm:ge3:claim:N(wi+)-on-Pn}
$N(w_i^+)\subseteq V(P_n)$ for $i=1,\dotsc,p-1$.
\end{claim}
\begin{proof}
If $w_i^+u\in E(G)$ for some $i=1,\dotsc,p-1$ and some $u\notin V(P_n)$,
then repeating \cref{thm:ge3:claim:N(wi),thm:ge3:claim:everysecond} with $u$ instead of $v$,
we get that $u$ is adjacent to every second vertex between $w_i^+$ and either $x$ or $y$.
But it is then impossible that $ux\in E(G)$, since $x\dir P_nw_i$ has an odd number of vertices,
which means that $u$ is adjacent to $y$ and, in particular, $w_{i+1}^+$.
But then there is an $\pathto xy$-path $x\dir P_nw_ivw_{i+1}w_i^+uw_{i+1}^+\dir P_ny$ with
$n+2$ vertices, a contradiction.
Thus $N(w_i^+)\subseteq V(P_n)$ for $i=1,\dotsc,p-1$.
\end{proof}

\begin{claim}
\label{thm:ge3:claim:WW+-complete}
$N(w_i^+)=W$ for $i=1,\dotsc,p-1$
and $W\subseteq N(w_p^+)$.
\end{claim}
\begin{proof}
Note that when proving \cref{thm:ge3:claim:N(wi),thm:ge3:claim:everysecond,thm:ge3:claim:t},
every time we reached a contradiction by constructing a longer $\pathto xy$-path,
the new path contained the vertex~$v$.
Also, note that $p\ge2$,
as otherwise $N(w_1^+)\cap N(v)=\set{w_1}$ by \cref{thm:ge3:claim:N(wi)},
contradicting the conditions of the theorem.
Now consider the path $P_n'=xvw_2\dir P_ny$.
Then \cref{thm:ge3:claim:N(wi),thm:ge3:claim:everysecond,thm:ge3:claim:t}
are valid for $P_n'$ with $x^+$ instead of $v$ as the outside vertex,
since otherwise we could construct an $\pathto xy$-path
containing all vertices of $V(P'_n)\cup\set{x^+}=V(P_n)\cup\set{v}$.
Note also that $t$ from \cref{thm:ge3:claim:t} has the property $t=\card{N(v)\cap N(x^+)}$,
so $t$ has the same value for $P_n$ and $v$ as for $P_n'$ and $x^+$.

We shall now prove that $t=p$.
Assume on the contrary that $t<p$ and let $W'=N(x^+)\cap V(P_n')$.
Since $x^+$ is adjacent to $t$~vertices in $W$,
it is easy to see that \cref{thm:ge3:claim:everysecond} for $P_n'$ and~$x^+$ implies that
$W'=\set{w_1,\dotsc,w_t}$.
It follows from \cref{thm:ge3:claim:t} for $P_n'$ and~$x^+$ that $w_t^+$ is adjacent to
all vertices in $W'$.
But then $\set{w_1,\dotsc,w_{t+1}}\subseteq N(w_t^+)\cap W$,
so $\card{N(w_t^+)\cap W}\ge t+1$,
contradicting \cref{thm:ge3:claim:t} for $P_n$ and~$v$.
We can conclude that $t=p$ and that $W\subseteq N(w_i^+)$ for $i=1,\dotsc,p$.

Now assume that $N(w_i^+)\ne W$ for some $i\in\set{1,\dotsc,p-1}$.
Then \cref{thm:ge3:claim:N(wi+)-on-Pn} implies that $w_i^+$ has a neighbor on $w_p^{++}P_ny$,
since $W^+$ is independent.
Now consider the path $P_n''=x\dir P_nw_ivw_{i+1}\dir P_ny$.
As above, \cref{thm:ge3:claim:N(wi),thm:ge3:claim:everysecond,thm:ge3:claim:t}
are valid for $P_n''$ with $w_i^+$ instead of $v$ as the outside vertex
and $t$ has the same value for $P_n$ and $v$ as for $P_n''$ and $w_i^+$.
Thus $w_p^{++}\in W''=N(w_i^+)\cap V(P_n'')$
by \cref{thm:ge3:claim:everysecond} for $P_n''$ and $w_i^+$.
This means that $\card{N(w_p^+)\cap W''}\ge\card{W\cup\set{w_p^{++}}}=t+1$, a contradiction.
We can conclude that $N(w_i^+)=W$ for $i=1,\dotsc,p-1$.
\end{proof}

We now know that $xw_p^+\in E(G)$, which means that $d(x,w_p^{++})\le2$.
This will be used to get our final contradiction.
If $xw_p^{++}\in E(G)$ then \cref{thm:ge3:claim:N(wi)} implies that $x^+w_p^{++}\in E(G)$,
contradicting \cref{thm:ge3:claim:WW+-complete}.
Thus $d(x,w_p^{++})=2$,
which means that
\begin{equation}
\label{eq:x-wp++-two}
\card{N(x)\cap N(w_p^{++})}\ge2.
\end{equation}
It follows from \cref{thm:ge3:claim:N(wi)} that
$N(x)\subseteq W^+\cup N(x^+)\cup N(v)\cup\set v$.
\Cref{thm:ge3:claim:WW+-complete} shows that $N(x^+)=W$,
and we shall see that $N(v)\cap N(x)\subset W$ as well.
Assume on the contrary that there is a vertex~$u\in N(x)\cap N(v)\setminus V(P_n)$.
Then $u\notin N(w_1^+)\cup N(w_2^+)$ by \cref{thm:ge3:claim:N(wi)}.
Thus
\begin{align}
p=\card{N(w_1^+)\cap N(w_2^+)}
&\ge\card[\big]{N(x)\setminus\bigl(N(w_1^+)\cup N(w_2^+)\bigr)}-1\notag\\
&\ge\card{W^+\cup\set{v,u}}-1\\
&=p+1,\notag
\end{align}
a contradiction.
We can conclude that $N(x)\subseteq W\cup W^+\cup\set v$.
\Cref{thm:ge3:claim:WW+-complete} implies that $N(w_p^{++})\cap W^+\subseteq\set{w_p^+}$,
and together with \cref{thm:ge3:claim:N(wi)} it implies that
$N(w_p^{++})\cap W\subseteq\set{w_p}$
since $w_p^{++}\notin W^+\cup N(v)\cup\set v$.
\Cref{eq:x-wp++-two} now implies that $N(x)\cap N(w_p^{++})=\set{w_p,w_p^+}$.
But then
\begin{align}
p=\card{N(x^+)\cap N(v)}
&\ge\card[\big]{N(w_p)\setminus\bigl(N(x^+)\cup N(v)\bigr)}-1\notag\\
&\ge\card{W^+\cup\set{v,w_p^{++}}}-1\\
&=p+1,\notag
\end{align}
our final contradiction.
The \lcnamecref{thm:pathextension-ge3} follows.

\subsection{Proof of \texorpdfstring{\cref{thm:infinite}}{Theorem~10}}
\label{sec:proofs:infinite}

To prove that the conditions of \cref{thm:cycleextension} are sufficient
to find a Hamilton curve,
we will use the following theorem by Kündgen, Li, and Thomassen,
along with an observation.
\begin{oldtheorem}[Kündgen--Li--Thomassen~\cite{kundgen17}]
\label{thm:kundgen}
The following are equivalent for any locally finite graph~$G$.
\begin{enumerate}
\item For every finite vertex set~$S$, $G$~has a cycle containing~$S$.
\item $\fcomp{G}$ has a Hamilton curve.
\end{enumerate}
\end{oldtheorem}

\begin{observation}
\label{obs:infinite}
In the proof of \cref{thm:cycleextension},
whenever we reach a contradiction by constructing a cycle~$C_{n+\ell}$,
the new cycle contains either the vertex~$v$ or a~vertex~$u$ at distance at most~$3$ from~$v$
(see \cref{rem:extension}).
Thus, if $G$ satisfies the conditions of \cref{thm:cycleextension}
and $v$~is a vertex adjacent to a cycle~$C_n$ in~$G$,
then there is a cycle $C_{n+\ell}$ containing all vertices of $C_n$
and at least one additional vertex from the set~$M_3(v)$,
unless $G\in\mathcal{K}$.
\end{observation}

Using \cref{obs:infinite,thm:kundgen} it is straightforward to prove \cref{thm:infinite}.
First note that $G\notin\mathcal{K}$ since is infinite.
Now for any finite vertex set~$S$,
pick a vertex~$a$ and an integer~$r$ such that $S\subseteq M_r(a)$,
and let $C_n$ be a cycle through $a$ containing as many vertices as possible from the set~$M_{r+3}(a)$.
If $C_n$ does not contain all vertices of $S$,
there is a vertex $v\in M_r(a)\setminus V(C_n)$ with a neighbor on~$C_n$,
and by using \cref{obs:infinite} we can find a cycle~$C_{n+\ell}$ containing more vertices of $M_{r+3}(a)$,
a contradiction.
Thus $C_n$ contains all vertices of~$S$.
Now, using \cref{thm:kundgen} we can conclude that $\fcomp{G}$ has a Hamilton curve.

\subsection{Proof of \texorpdfstring{\cref{thm:bipartite}}{Theorem~12}}
\label{sec:proofs:bipartite}

For a non-regular graph~$G$ with maximum degree at least three it is straight\-forward
to use \cref{lem:bipartite:smalldiff} below to prove that
$G$ is a subgraph of a complete bipartite graph~$K_{n,n}$ with a single vertex or a single edge removed,
by simply constructing the possible graphs vertex by vertex.
Similarly one can prove, using \cref{lem:bipartite:regulard=2}, that
every regular, connected, bipartite $L_1$-graph with maximum degree at least three
is either a complete bipartite graph $K_{n,n}$ or a subgraph of~$K_{n,n}$ with a perfect matching removed.
\Cref{thm:bipartite} follows.
For details, see~\cite{granholm16}.

\begin{observation}[\cite{granholm16}]
\label{lem:bipartite:smalldiff}
Let $G$ be a bipartite $L_1$-graph
and let $u$ and $v$ be two adjacent vertices in~$G$.
Then $\abs{d(u)-d(v)}\le1$.
\end{observation}

\begin{observation}[\cite{granholm16}]
\label{lem:bipartite:regulard=2}
Let $G$ be an $n$-regular bipartite $L_1$-graph
and let $u$ and~$v$ be two vertices at distance~$2$ in~$G$.
Then $\card{N(u)\cap N(v)}\ge n-1$.
\end{observation}

\section*{Acknowledgement}
I would like to thank Armen Asratian and Carl Johan Casselgren
for many helpful comments and fruitful discussions while preparing this work.

\end{document}